\documentclass{amsart}
\usepackage{amssymb}
\usepackage{amsfonts}

\usepackage{mathtools,thmtools}
\mathtoolsset{centercolon}
\usepackage{enumitem}
\setlist[enumerate]{label=\upshape(\roman*)}
\usepackage[T1]{fontenc}
\usepackage[utf8]{inputenc}
\usepackage[hidelinks]{hyperref}
\DeclareUnicodeCharacter{012D}{\u\i}
\hypersetup{
  pdftitle={Kadison's Pythagorean Theorem and essential codimension},
  pdfauthor={Victor Kaftal and Jireh Loreaux}
}

\usepackage{nicefrac}
\usepackage{latexsym}
\usepackage[mathscr]{eucal}
\usepackage{verbatim}
\usepackage{color}

\theoremstyle{plain}
\newtheorem{theorem}{Theorem}[section]

\newtheorem{lemma}[theorem]{Lemma}
\newtheorem{proposition}[theorem]{Proposition}
\newtheorem{remark}[theorem]{Remark}
\newtheorem{definition}[theorem]{Definition}
\newtheorem{example}[theorem]{Example}

\expandafter\mathchardef\expandafter\varphi\number\expandafter\phi\expandafter\relax
\expandafter\mathchardef\expandafter\phi\number\varphi

\newcommand{\B}{\mathcal{B}}

\newcommand{\K}{\mathcal{K}}
\newcommand{\h}{\mathcal{H}}
\newcommand{\eps}{\epsilon}

\DeclareMathOperator{\spans}{span}

\DeclareMathOperator{\tr} {Tr}

\DeclareMathOperator{\M}{\mathcal M(\mathcal A \otimes \mathcal K)}

\DeclareMathOperator{\idx}{ind}
\DeclareMathOperator{\diag}{diag}

\newcommand{\projone}{p}
\newcommand{\projtwo}{q}

\numberwithin{equation}{section}

\title[Kadison's Pythagorean Theorem and essential codimension]{Kadison's Pythagorean Theorem \\ and essential codimension}

\author{Victor Kaftal}

\address{Department of Mathematics\\
University of Cincinnati\\
P. O. Box 210025\\
Cincinnati, OH\\
45221-0025\\
USA}

\email{victor.kaftal@uc.edu}

\author{Jireh Loreaux}

\address{Department of Mathematics and Statistics\\
Southern Illinois University Edwardsville \\
1 Hairpin Dr \\
Edwardsville, IL \\
62026-1653 \\
USA}

\email{jloreau@siue.edu}

\keywords{Essential codimension, Fredholm pairs of projections, diagonals of projections, diagonals of selfadjoint operators}
\subjclass[2010]{Primary:47B15, 47A53; Secondary: 46C05, 42C15.}

\thanks{This work was partially supported by the Simons Foundation grant No 245660 to Victor Kaftal}

\begin{document}

\begin{abstract}

Kadison's Pythagorean theorem (2002) provides a characterization of the diagonals of projections with a subtle integrality condition. Arveson (2007), Kaftal, Ng, Zhang (2009), and Argerami (2015) all provide different proofs of that integrality condition. In this paper we interpret the integrality condition in terms of the essential codimension of a pair of projections introduced by Brown, Douglas and Fillmore (1973), or, equivalently of the index of a Fredholm pair of projections introduced by Avron, Seiler, and Simon (1994). The same techniques explain the integer occurring in the characterization of diagonals of selfadjoint operators with finite spectrum by Bownik and Jasper (2015).

\end{abstract}
\maketitle
\section {Introduction}

In his seminal papers on the Pythagorean Theorem (\cite{Kad-2002-PNASU,Kad-2002-PNASUa}), Kadison characterizes the diagonals of projections, that is the sequences that can appear on the diagonal of a matrix representation of a projection.
The main assertion of his Theorem 15 is by now usually paraphrased as follows:

\begin{theorem}[\protect{\cite[Theorem 15]{Kad-2002-PNASUa}}]
  \label{T:Kadison15}
  A sequence $\{d_n\}$ with $0\le d_n\le 1$ is the diagonal of a projection $B(\mathcal{H})$ if and only if for
  \begin{equation*}
    a = \sum_{d_n \le \nicefrac{1}{2}} d_n
    \qquad\text{and}\qquad
    b = \sum_{d_n > \nicefrac{1}{2}} (1-d_n),
  \end{equation*}
  either
  \begin{enumerate}
  \item $a+b = \infty$, or
  \item $a+b < \infty$ and $a-b \in \mathbb{Z}$.
  \end{enumerate}
\end{theorem}
Kadison proved that $a-b$ is arbitrarily close to an integer and hence is an integer and referred to that integer as ``curious''.

Let us first express $a$ and $b$ in operator theoretic terms.
Call $\projone{}$ the projection, $\{e_j\}$ the orthonormal basis of $\h$ used for the matrix representation, and $\projtwo{}$ the projection on $\overline {\spans}\{e_j\mid~ d_j> \nicefrac{1}{2}\}.$
Then
\begin{equation}
  \label{e:ab}
  a = \tr(\projtwo{}^\perp \projone{} \projtwo{}^\perp) \quad \text{and}\quad b = \tr( \projtwo{}- \projtwo{}\projone{}\projtwo{}),
\end{equation}
hence if $a+b<\infty$, we have
$\projtwo{}^\perp (\projone{}-\projtwo{}) \projtwo{}^\perp=\projtwo{}^\perp \projone{} \projtwo{}^\perp \in \mathcal L^1$, $\projtwo{}(\projone{}-\projtwo{})\projtwo{}=-(\projtwo{}- \projtwo{}\projone{}\projtwo{})  \in \mathcal L^1$, and
\begin{equation}
  \label{e:a-b}
  a-b = \tr\big( \projtwo{}(\projone{}-\projtwo{})\projtwo{} + \projtwo{}^\perp (\projone{}-\projtwo{}) \projtwo{}^\perp \big).
\end{equation}
If we knew that $\projone{}-\projtwo{}\in \mathcal L^1$, then we would have $a-b=\tr(\projone{}-\projtwo{})$ and then, by \cite[Lemma 4.1]{Eff-1989-MI}, we could conclude that $\tr(\projone{}-\projtwo{})\in\mathbb Z$.
However, since $\projone{}-\projtwo{}$ is not necessarily positive, the fact that its corners are trace-class does not imply that $\projone{}-\projtwo{}$ itself is trace-class.
In fact, Argerami proved in \cite{Arg-2015-IEOT}  that $\projone{}-\projtwo{}\in \mathcal L^2$ and by modifying Effros' argument, he showed that this is sufficient to guarantee that $a-b$ is an integer.
However, neither Kadison's nor Argerami's proof shed much light on the origin of that integer itself.

One of Bill Arveson's sayings was that if you find an integer in operator theory you should look for a Fredholm operator.
Arveson partially extended Kadison's work on the Pythagorean Theorem in \cite{Arv-2007-PNASU} where he studied the diagonals of normal operators with finite spectrum with infinite multiplicity that forms the vertices of a convex polygon in $\mathbb{C}$, infinite co-infinite projections being a degenerate special case.
He also found an ``index obstruction'' for their diagonals which depended on the following result.
\begin{theorem}[\protect{\cite[Theorem 3]{Arv-2007-PNASU}}]
  \label{T:Arveson}
  Let $\projone{},\projtwo{}$ be projections in $\B(\h)$ with $\projone{}-\projtwo{}\in \mathcal L^2$.
  Then $\projone{}\wedge \projtwo{}^\perp$ and $\projone{}^\perp \wedge \projtwo{}$ are finite, $\projtwo{}(\projone{}-\projtwo{})\projtwo{}$ and $\projtwo{}^\perp(\projone{}-\projtwo{})\projtwo{}^\perp$ belong to $\mathcal L^1$, and
  \begin{equation*}
    \tr\big(\projtwo{}(\projone{}-\projtwo{})\projtwo{}+\projtwo{}^\perp(\projone{}-\projtwo{})\projtwo{}^\perp\big)= \tr(\projone{}\wedge \projtwo{}^\perp)-\tr(\projone{}^\perp \wedge \projtwo{})\in \mathbb Z.
  \end{equation*}
\end{theorem}

Whenever we have two projections, $\projone{}$ and $\projtwo{}$, we denote by $\projtwo{}\mid_{\projone{}\h}$ the operator in $\B(\projone{}\h, \projtwo{}\h)$.
Then a key step in Arveson's proof is the fact that if $\projone{}-\projtwo{}\in \mathcal L^2$, then
\begin{equation}
  \label{e: Arveson index}
  \projtwo{}\mid_{\projone{}\h}\quad\text{is Fredholm and}\quad \idx(\projtwo{}\mid_{\projone{}\h})= \tr\big(\projtwo{}(\projone{}-\projtwo{})\projtwo{}+\projtwo{}^\perp(\projone{}-\projtwo{})\projtwo{}^\perp\big).
\end{equation}

Although Arveson did not state so explicitly, embedded in his proofs one can also find the fact that using the notations established above, if $a+b< \infty$, then indeed $\projone{}-\projtwo{}\in \mathcal L^2$ and hence $a-b= \idx(\projtwo{}\mid_{\projone{}\h})$, which explains why $a-b$ is an integer.
What remains to be explained is the role of $\projtwo{}\mid_{\projone{}\h}$ and the significance of its index.
Note that $\idx (\projtwo{}\mid_{\projone{}\h}) = -\idx (\projone{}\mid_{\projtwo{}\h})$ since $(\projtwo{}\mid_{\projone{}\h})^{*} = \projone{}\mid_{\projtwo{}\h}$.

A similar question arises from another proof that $a-b$ is an integer which was obtained in \cite[Corollary 3.6]{KNZ-2009-JFA}.
Let us briefly sketch the original computation (reformulated in new notation) as it introduces the connections we want to illustrate.

Let $w$ be an isometry with range $\projone{}$, let $\Lambda:=\{j\mid d_j> \nicefrac{1}{2}\}$,
\begin{align*}
  f_j                                & :=
       \begin{cases}
         \frac{1}{\sqrt {d_j}}w^*e_j & d_j\ne 0 \\
         e_1                         & d_j=0    \\
       \end{cases}                                                                      \\
  f                                  & : = \sum_{j\in \Lambda} e_j\otimes f_j           \\
  f                                  & =v|f| \quad\text{the polar decomposition of $f$} \\
  t_a                                & :=  \sum_{j\not \in \Lambda}d_j f_j\otimes f_j   \\
  t_b                                & :=  \sum_{j \in \Lambda}(1-d_j) f_j\otimes f_j
\end{align*}
Then $\|f_j\|=1$ for all $j$, $t_a, t_b\in \mathcal L^1_+$, and $1 = \sum_j d_j f_j \otimes f_j$, hence
\begin{align*}
  t_a-t_b   & = 1- f^*f= v^*v(t_a-t_b)+ 1-v^*v \\
  E(\projtwo{}-vv^*) & = - E(v(t_a-t_b)v^*).
\end{align*}
Hence
\begin{equation*}
  a-b=\tr(t_a-t_b)= - \tr (\projtwo{}-vv^*) + \tr(1-v^*v)\in \mathbb Z.
\end{equation*}

It is then immediate to see (but was not remarked explicitly in \cite{KNZ-2009-JFA}), that
\begin{equation*}
  a-b = - \idx (v^*\mid_{\projtwo{}\h}).
\end{equation*}
Notice that $f$ can be interpreted as the frame transform (the analysis operator) of the Bessel sequence $\{f_j\}_{j\in \Lambda}$ and $v$ as the frame transform of the associated Parseval frame.
While this construction provides indeed a proof that $v^*\mid_{\projtwo{}\h}$ is Fredholm, a natural question is why $ \idx (v^*\mid_{\projtwo{}\h})=  \idx (\projone{}\mid_{\projtwo{}\h})$ as can be obtained from Arveson's work.
To answer it, notice first that since $f^*f$ is a trace-class perturbation of the identity, it is Fredholm and hence so are $|f|$ and $f^*=|f|v^*$.
Furthermore $ \idx (f^*\mid{\projtwo{}\h})= \idx(v^*\mid{\projtwo{}\h})$ since $\idx(|f|)=0$.
Next
\begin{equation*}
  w^*\projtwo{}= \sum _{j\in \Lambda} w^*e_j\otimes e_j=  \sum _{j\in \Lambda}\sqrt{d_j} f_j\otimes e_j = fd
\end{equation*}
where $d:=  \sum _{j\in \Lambda}\sqrt{d_j}e_j\otimes e_j\ge \frac{1}{\sqrt{2}} \projtwo{}$ is invertible in $\B(\projtwo{}\h)$.
Thus $w^*\mid_{\projtwo{}\h}$ is also Fredholm in $\B(\projtwo{}\h, \h)$ and
\begin{equation}
  \label{e:a-b via w*}
  a-b = -\idx (w^*\mid_{\projtwo{}\h}).
\end{equation}
It is then immediate to verify (see also (\ref{e:equiv}) below) that $ \idx (w^*\mid_{\projtwo{}\h})=  \idx (\projone{}\mid_{\projtwo{}\h})$ as obtained by Arveson.

However, neither the proof due to Arveson nor the one in \cite{KNZ-2009-JFA}) provides a natural explanation of the role of $w^*\mid_{\projtwo{}\h}$ or $\projone{}\mid_{\projtwo{}\h}.$

The goal of our paper is to provide an explanation of that role in the context of the notion of essential codimension $[\projone{}:\projtwo{}]$ of a pair of projections $\projone{}$ and $\projtwo{}$ with $\projone{}-\projtwo{}\in \K$ that was introduced in the BDF theory (see \cite{BDF-1973-PoaCoOT} and Section 2), or of the more general notion of index of a Fredholm pair of projections, introduced by Avron, Seiler, and Simon in \cite{ASS-1994-JFA}.

Combining Arveson's work with the study of Fredholm pairs and essential codimension, one can provide a natural identification of Kadison's integer with the essential codimension of a pair of projections.
In the notations of \autoref{T:Kadison15} we have:
\begin{theorem}
  \label{T:ess codim}
  Let $\projone{}\in \B(\h)$ be a projection such that $a+b< \infty$ and let $\projtwo{}$ be the projection on $\overline {\spans}\{e_j\mid~ d_j> \nicefrac{1}{2}\}$.
  Then $\projone{}-\projtwo{}\in \mathcal L^2$ and $a-b = [\projone{}:\projtwo{}].$
\end{theorem}

To understand the simple proof of this result, and for the convenience of the readers not familiar with the notions of Fredholm pairs, essential codimension, and the work of Arveson in \cite{Arv-2007-PNASU}, we will provide in Section 2 a self-contained short presentation of the relevant results of the theory of Fredholm pairs.
We have strengthened several results and generalized them to the case when $\projone{}-\projtwo{}$ belongs to an arbitrary (two-sided) operator ideal $\mathcal J$ rather than just the Hilbert--Schmidt ideal $\mathcal L^2$.

Since Fredholm pairs have found most of their applications in the theory of spectral flows in type I or type II von Neumann algebras, we will conclude Section 2 with a very brief foray into the case when the notion of Fredholm operators and indices are taken relative to a semifinite von Neumann algebra (also called Breuer--Fredholm, or more precisely $\tau$-Breuer--Fredholm operators).

In Section 3 we will assemble the results previously collected into a proof of \autoref{T:ess codim} that is inspired by, but independent of, the work by Arveson.
Then we will extend part of \autoref{P:corners} to positive contractions.
We will then use the same techniques to identify an integer appearing in the study by Bownik and Jasper of the diagonals of selfadjoint operators with finite spectrum and also to simplify the proof of one of the key results of that paper (see \cite{BJ-2013-TAMS}).

We thank R.~Douglas for having suggested to the first named author of this paper to consider a possible connection between the frame transform approach originally used and essential codimension.

\section{Essential codimension and Fredholm pairs}

In this paper $\h$ denotes a separable infinite dimensional complex Hilbert space and
$\K$ the C*-algebra of compact operators on $\h$.

The notion of essential codimension of two projections was first introduced
in (\cite[Remark 4.9]{BDF-1973-PoaCoOT}).
\begin{definition}[\protect{\cite{BDF-1973-PoaCoOT}}]
  \label{D:esscodim}
  Given projections $\projone{},\projtwo{} \in\B(\h)$ for which $\projone{}-\projtwo{} \in \K(\h)$, the \emph{essential codimension} $[\projone{}:\projtwo{}]$ of $\projone{}$ and $\projtwo{}$ is defined by:
  \begin{equation*}
    [\projone{}:\projtwo{}] :=
    \begin{cases}
      \tr(\projone{})-\tr(\projtwo{}) & \tr(\projone{})<\infty, \tr(\projtwo{})<\infty                                           \\
      \idx(v^{*}w)    & \tr(\projone{})= \tr(\projtwo{})= \infty, \, w,\,v\text{ isometries},\,ww^*= \projone{}, vv^*=\projtwo{}. \\
    \end{cases}
  \end{equation*}
\end{definition}
This definition depends on the fact that $(v^*w)^*(v^*w)= 1+ w^*(\projtwo{}-\projone{})w$ and similarly for $w^*v$.
Thus setting $\pi$ to be the projection onto the Calkin algebra, we see that $\pi(v^*w)$ is unitary and hence $v^*w$ is Fredholm.
Also, if $\tilde w $ and $\tilde v$ are another pair of isometries with ranges $\projone{}$ and $\projtwo{}$ respectively, then
\begin{equation*}
  \idx(v^*w)= \idx( v^*\tilde v \tilde v^* \tilde w \tilde w^* w)= \idx(\tilde v^*\tilde w)
\end{equation*}
since $w^*\tilde w$ and $v^*\tilde v$ are unitaries.
This shows that $[\projone{}:\projtwo{}]$ does not depend on the choice of the isometries $w$ and $v$.

Some properties of the essential codimension were presented without proof in \cite{Bro-1981-JOT} and a complete exposition can be found in \cite{BL-2012-CJM}, together with an interesting application to liftability of projections in the corona algebra of, among others, $C([0,1])\otimes \K$.

Independently, and without reference to essential codimension, Avron, Seiler, and Simon defined in \cite{ASS-1994-JFA} the more general notion of Fredholm pairs of projections.
\begin{definition}[\protect{\cite{ASS-1994-JFA}}]
  \label{D:Fredhpair}
  A pair of projections $(\projone{}, \projtwo{})$ in $\B(\h)$ is said to be Fredholm if $\projtwo{}\mid_{\projone{}\h}$ is a Fredholm operator as an element of $\B(\projone{}\h, \projtwo{}\h)$.
  The index of the pair is defined to be $ \idx(\projtwo{}\mid_{\projone{}\h})$.
\end{definition}

Notice that if $v\in \B(\h_4,\h_3)$ and $w\in \B(\h_1,\h_2)$ are unitaries, then
\begin{equation}
  \label{e: Fredh diff spaces}
  g\in \B(\h_2,\h_3) ~\text { is Fredholm } ~\Leftrightarrow ~ v^*gw\in \B(\h_1, \h_4) ~\text{ is Fredholm.}
\end{equation}

and then
\begin{equation}
  \label{e: Fredh as compression}
  \idx(v^*gw)= \idx g.
\end{equation}
Thus if $w$ and $v$ are isometries with ranges $\projone{}$ and $\projtwo{}$ respectively, then
\begin{align}
  \label{e:equiv}
  \projtwo{}\mid_{\projone{}\h}\in \B(\projone{}\h, \projtwo{}\h)\text { is Fredholm }
  & ~\Leftrightarrow  ~ v^*w = v^*\mid_{\projone{}\h}w\in \B(\h) \text { is Fredholm} \\
  & ~\Leftrightarrow  ~v^*\mid_{\projone{}\h}\in \B(\projone{}\h, \h) \text { is Fredholm}\notag
\end{align}

Recall that $v^{*}w$ is Fredholm if and only if $\pi(v^{*}w)$ is invertible.
We have seen above that if $\projone{}-\projtwo{}\in \K$, then $\pi(v^{*}w)$ is unitary and hence $v^*w$ is Fredholm, that is $(\projone{},\projtwo{})$ is a Fredholm pair and by (\ref{e: Fredh as compression}), $[\projone{}:\projtwo{}]= \idx(\projtwo{}\mid_{\projone{}\h})$.
For consistency we will henceforth write $[\projone{}:\projtwo{}]:=\idx(\projtwo{}\mid_{\projone{}\h})$ whenever $(\projone{},\projtwo{})$ is a Fredholm pair even when $\projone{}-\projtwo{}\not \in \K$.

Soon after \cite{ASS-1994-JFA}, W.~Amrein, K.~Sinha \cite{AS-1994-LAA} realized that the proofs in \cite{ASS-1994-JFA} could be considerably simplified by reducing to the case of projections in generic position.
This notion was first introduced by Dixmier \cite{Dix-1948-RS} (he called them in ``position {\it p}'') and independently by Krein, Kranosleskii and Milman \cite{KKM-1948-STIMANS}, and further studied by Davis \cite{Dav-1958-ASMS}, Halmos\cite{Hal-1969-TAMS} (he called them ``generic pairs''), and others.

\begin{definition}
  \label{D:genericposition}
  Two projections $\projone{}, \projtwo{}\in \B(\h)$ are said to be in \emph{generic position} if
  \begin{equation*}
    \projone{} \wedge \projtwo{} = \projone{} \wedge \projtwo{}^{\perp} = \projone{}^{\perp} \wedge \projtwo{} = \projone{}^{\perp} \wedge \projtwo{}^{\perp}=0.
  \end{equation*}
\end{definition}
When just the first three projections are zero, the pair $\projone{},\projtwo{}$ is in generic position in $\B((\projone{}\vee \projtwo{})\h)$ and when there is no risk of confusion we will simply call them in generic position.
For the readers' convenience we will collect here some results on projections in generic position.
Good references can be found in the texts of Str\u{a}til\u{a}  \cite[17.15]{Str-1981}, Takesaki \cite[pages 306-308]{Tak-2002-TOOAI} and in the article of Amrein and Sinha \cite{AS-1994-LAA}.
\begin{theorem}
  \label{T:genericposition}
  Let $\projone{}, \projtwo{}\in \B(\h)$ be two projections.
  \begin{enumerate}
  \item Then the projections $\projone{}_0:= \projone{}- \projone{}\wedge \projtwo{}- \projone{}\wedge \projtwo{}^\perp$ and $\projtwo{}_0:= \projtwo{} - \projone{}\wedge \projtwo{}- \projone{}^\perp\wedge \projtwo{}$ are in generic position (in the Hilbert space $(\projone{}_0\vee \projtwo{}_0)\h$.)
  \end{enumerate}
  Suppose $\projone{}, \projtwo{}$ are in generic position and let $N:=\{\projone{},\projtwo{}\}''$ be the von Neumann algebra generated by them.
  Then
  \begin{enumerate}[resume]
  \item $N_\projtwo{} (=(\projtwo{}N\projtwo{}\mid_{\projtwo{}\h})$ is masa of $N$ and $N$ can be identified with $M_2(N_\projtwo{})$.
  \item There are two positive injective contractions $c$ and $s$ in $N_\projtwo{}$ with $c^2+s^2=1$ (the identity operator of $N_\projtwo{}$) such that $\projone{}$ and $\projtwo{}$ can be identified with
    \begin{equation}
      \label{e:p,q}
      \projtwo{} =
      \begin{pmatrix}
        1 & 0 \\
        0 & 0 \\
      \end{pmatrix}
      \quad\text{and}\quad
      \projone{} =
      \begin{pmatrix}
        c^2 & cs \\
        cs & s^2 \\
      \end{pmatrix}.
    \end{equation}
  \item $c=(\projtwo{}\projone{}\projtwo{})^{\nicefrac{1}{2}}\mid_{\projtwo{}\h}$, $s=(\projtwo{}\projone{}^\perp \projtwo{})^{\nicefrac{1}{2}}\mid_{\projtwo{}\h}$, and $\|s\|=\|\projone{}-\projtwo{}\|$.
  \end{enumerate}
\end{theorem}
In this section we will often use the representation (\ref{e:p,q}) without further reference.
Notice that for projections in generic position, the equality $\|s\|=\|\projone{}-\projtwo{}\|$ (\cite[17.15 (8)]{Str-1981}, see also \cite[pg 391]{Dix-1948-RS} and \cite[pg 308]{Tak-2002-TOOAI}) follows also from the identity:
\begin{equation}
  \label{e:q-p}
  \projone{}-\projtwo{}=
  \begin{pmatrix}
    -s^2 & cs \\
    cs & s^2 \\
  \end{pmatrix}
  =
  \begin{pmatrix}
    -s & c \\
    c & s \\
  \end{pmatrix}
  \begin{pmatrix}
    s & 0 \\
    0 & s \\
  \end{pmatrix}
  =
  \begin{pmatrix}
    s & 0 \\
    0 & s \\
  \end{pmatrix}
  \begin{pmatrix}
    -s & c \\
    c & s \\
  \end{pmatrix}
\end{equation}
Notice also that if we set
\( v: = \frac{1}{\sqrt{2}}
\begin{pmatrix}
  (1-s)^{\nicefrac{1}{2}} & (1+s)^{\nicefrac{1}{2}}  \\
  (1+s)^{\nicefrac{1}{2}} & -(1-s)^{\nicefrac{1}{2}} \\
\end{pmatrix}
\), then $v=v^*$ is unitary and
\begin{equation}
  \label{e:diag q-p}
  \projone{}-\projtwo{}= v
  \begin{pmatrix}
    s & 0  \\
    0 & -s \\
  \end{pmatrix}
  v^*.
\end{equation}

It is well known that projections in generic position are unitarily equivalent in $N$, and over the years various authors (e.g., \cite{Kat-1966}, \cite{BL-2012-CJM}, \cite{Dix-1948-RS}, \cite{ASS-1994-JFA}) have constructed different unitaries in $N$ implementing the equivalence.
We will use the following unitary:
\begin{equation}
  \label{e: unit equiv}
  \projone{}=u\projtwo{}u^*\quad \text{for the unitary }u:=
  \begin{pmatrix}
    c & -s \\
    s & c  \\
  \end{pmatrix}.
\end{equation}

As shown by Amrein and Sinha in \cite{AS-1994-LAA}, reduction to generic position makes the analysis of Fredholm pairs simpler and more transparent.
For the convenience of the readers, we will provide here a short self-contained presentation of the main results on Fredholm pairs that we will need in the sequel, completing and generalizing results obtained in \cite{BL-2012-CJM}, \cite{ASS-1994-JFA}, \cite{AS-1994-LAA}, \cite{Arv-2007-PNASU}.
The starting point is the analysis of the case when the projections are in generic position.
Recall that Fredholm operators are characterized by being invertible modulo the compact operators.
\begin{lemma}
  \label{L:invert}
  Let $\projone{},\projtwo{}$ be projections in generic position.
  Then the following are equivalent.
  \begin{enumerate}
  \item $\projone{}\mid_{\projtwo{}\h}$ is a Fredholm operator
  \item $c$ is invertible
  \item $\projone{}\mid_{\projtwo{}\h}$ is invertible
  \item $\|\projone{}-\projtwo{}\|< 1$
  \item $\|\projone{}-\projtwo{}\|_{ess}< 1$.
  \end{enumerate}
  If these conditions are satisfied, then $\idx(\projone{}\mid_{\projtwo{}\h})=0$.
\end{lemma}
\begin{proof}

(i)$~\Rightarrow~$(ii).
Since $|\projone{}\mid_{\projtwo{}\h}|=c$ is injective, and since $\projone{}\mid_{\projtwo{}\h}$ has closed range, by the Inverse Mapping Theorem $\projone{}\mid_{\projtwo{}\h}$ is invertible and hence so is $c$.\\
(ii)$~\Rightarrow~$(iii).
Immediate since $\projone{}\mid_{\projtwo{}\h}= \begin{pmatrix}c^2\\cs\end{pmatrix} $ and $(1, c^{-1}s)= \begin{pmatrix}c^2\\cs\end{pmatrix}^{-1}$.\\
(iii) $~\Rightarrow~$ (i).
Obvious.\\
(ii)$~\Leftrightarrow~$(iv).
Immediate since $\|\projone{}-\projtwo{}\|=\|s\|< 1$ and $c$ is invertible iff $c^2\ge \delta 1$ for some $\delta>0$ iff $s^2\le (1-\delta)1$ for some $\delta>0$.
\\
(iv)$~\Leftrightarrow~$(v).
By (\ref{e:diag q-p}), $\|\projone{}-\projtwo{}\|_{ess}= \|s\|$ and since $s$ is positive and injective this implies that $\|\projone{}-\projtwo{}\|=\|s\|<1$.
The other direction is trivial.
\end{proof}
The equivalence of (i) and (iv) and the fact that then the index is zero were obtained in \cite[Proposition 4]  {AS-1994-LAA}.
Using this lemma it is now easy to obtain a characterization of Fredholm pairs also when the projections are not in generic position.
\begin{proposition}
  \label{P: charact Fredh pairs}
  Let $\projone{}, \projtwo{}$ be projections in $\B(\h)$.
  Then $(\projone{},\projtwo{})$ is a Fredholm pair if and only if $\|\projone{}-\projtwo{}\|_{ess}<1$ and then $\projone{}\wedge \projtwo{}^\perp$ and $\projone{}^{\perp} \wedge \projtwo{}$ are finite and
  \begin{equation*}
    [\projone{}:\projtwo{}]= \tr(\projone{}\wedge \projtwo{}^\perp)-\tr(\projone{}^{\perp} \wedge \projtwo{}).
  \end{equation*}
\end{proposition}
\begin{proof}

In the notations of \autoref{T:genericposition}  and (\ref{e:q-p}) and (\ref{e:diag q-p}) we have
\begin{equation*}
  \projone{}= (\projone{}\wedge \projtwo{} + \projone{}\wedge \projtwo{}^\perp )  \oplus \projone{}_0\text{ and   }
  \projtwo{}= (\projone{}\wedge \projtwo{}+ \projone{}^{\perp} \wedge \projtwo{}) \oplus \projtwo{}_0
\end{equation*}
and hence
\begin{equation}
  \label{e:q-p Perp}
  \projone{}-\projtwo{}= (\projone{}\wedge \projtwo{}^\perp- \projone{}^{\perp} \wedge \projtwo{})  \oplus (\projone{}_0-\projtwo{}_0).
\end{equation}

As in \cite[Theorem 2] {AS-1994-LAA}, it is easy to see that $(\projone{},\projtwo{})$ is a Fredholm pair if and only if both $\big(\projone{}\wedge \projtwo{}+  \projone{}\wedge \projtwo{}^\perp, \projone{}\wedge \projtwo{}+\projone{}^{\perp}\wedge \projtwo{}\big)$ and $(\projone{}_0,\projtwo{}_0)$ are Fredholm pairs, and then the index of $(\projone{},\projtwo{})$ is the sum of the indices of the other two pairs.
It is obvious that the first pair is Fredholm if and only if both $\projone{} \wedge \projtwo{}^{\perp}$ and $\projone{}^{\perp}\wedge \projtwo{}$ are finite and then its index is $\tr(\projone{} \wedge \projtwo{}^{\perp})-\tr(\projone{}^{\perp}\wedge \projtwo{})$.
By \autoref{L:invert}, the pair $(\projone{}_0,\projtwo{}_0)$ is Fredholm if and only if $\|\projone{}_0-\projtwo{}_0\|_{ess}< 1$ and then its index is zero.
Thus to conclude the proof one just notices that $\projone{}-\projtwo{}$ is a finite rank perturbation of $\projone{}_0-\projtwo{}_0$ and hence it has the same essential norm.
\end{proof}

The implication that if $(\projone{},\projtwo{})$ is a Fredholm pair then $\projone{} \wedge \projtwo{}^{\perp}$ and $\projone{}^\perp \wedge \projtwo{}$ are finite and the formula for the index of the pair was obtained in \cite[Theorem 2] {AS-1994-LAA}.
The necessity and sufficiency of the condition $\|\projone{}-\projtwo{}\|_{ess}<1$ (albeit not expressed in terms of essential norm) was only implicit in \cite{ASS-1994-JFA}, and was obtained explicitly and with more generality in \cite{BCP+-2006-Agatoeo}.

We consider now the cases when the difference $\projone{}-\projtwo{}$ belongs to some proper operator ideal $\mathcal J$.

\begin{proposition}
  \label{P:in J}
  Let $\mathcal J$ be an operator ideal and $\projone{},\projtwo{} \in B(\h)$ be projections.
  Then
  \begin{enumerate}
  \item $\projone{}-\projtwo{} \in \mathcal J$ if and only if $\projone{}^\perp \wedge \projtwo{}$ and $\projone{} \wedge \projtwo{}^\perp$ are finite and $s \in \mathcal J$ (where
    \( \projone{}_0:= \projone{}- \projone{}\wedge \projtwo{}- \projone{} \wedge \projtwo{}^\perp =
    \begin{pmatrix}
      c^2 & cs \\
      cs & s^2 \\
    \end{pmatrix}
    \) as by (\ref{e:p,q})).
  \item If $\projone{}-\projtwo{}\in \mathcal J$, then $[\projone{}:\projtwo{}]=0$ if and only if there is a unitary $u\in 1+\mathcal J$ such that $u\projtwo{}u^*=\projone{}$.
  \end{enumerate}
\end{proposition}
\begin{proof}
  (i). By (\ref{e:q-p Perp}) we see that $\projone{}-\projtwo{}\in \mathcal J$ if and only if $\projone{} \wedge \projtwo{}^\perp \in \mathcal J$, $\projone{}^\perp \wedge \projtwo{}\in \mathcal J$ and $\projone{}_0-\projtwo{}_0\in \mathcal J$.
  By (\ref{e:diag q-p}), the last condition holds if and only if $s \in \mathcal J$, and the conclusion then follows from the fact that a projection belongs to a proper ideal if and only if it is finite.

  (ii).  First notice that from part (i) and \autoref{P: charact Fredh pairs} it follows that
  \begin{equation}
    \label{e:interm implic}
    \projone{}-\projtwo{}\in \mathcal J \text{ and } [\projone{}:\projtwo{}]=0~\Leftrightarrow~ \tr (\projone{} \wedge \projtwo{}^\perp) = \tr (\projone{}^\perp \wedge \projtwo{})< \infty\text{ and }s\in \mathcal J.
  \end{equation}

  Therefore if there is a unitary in $1+ \mathcal J$ such that $u\projtwo{}u^*=\projone{}$ , then it is immediate to see that $\projone{}-\projtwo{}\in \mathcal J$ and $[\projone{}:\projtwo{}]=0$.

  Conversely, assume that $\projone{}-\projtwo{}\in \mathcal J$ and $[\projone{}:\projtwo{}]=0$.
  Then by (\ref{e:interm implic}) $s\in \mathcal J$, $\projone{}^\perp \wedge \projtwo{} + \projone{} \wedge \projtwo{}^\perp$ is finite, and $\tr (\projone{} \wedge \projtwo{}^\perp) = \tr (\projone{}^\perp \wedge \projtwo{})$ and hence $\projone{} \wedge \projtwo{}^\perp \sim \projone{}^\perp \wedge \projtwo{}$.
  Then there is a unitary $u_1$ on $\big(\projone{} \wedge \projtwo{}^\perp + \projone{}^\perp \wedge \projtwo{})\h$ such that $u_1(\projone{}^\perp \wedge \projtwo{}) u_1^*= \projone{} \wedge \projtwo{}^\perp$.

  Let
  $u_2:=\begin{pmatrix}c&-s\\s&c\end{pmatrix}$.
  Then by (\ref{e: unit equiv}), $u_2$ is a unitary on $(\projone{}_0\vee \projtwo{}_0)\h)$ and $\projone{}_0=u_2\projtwo{}_0u_2^*$.
  Furthermore,
  \begin{equation*}
    u_2-1\mid_{(\projone{}_0\vee \projtwo{}_0)\h} = \begin{pmatrix}c-1&-s\\s&c-1\end{pmatrix}\in \mathcal J
  \end{equation*}
  because
  \begin{equation*}
    c-1= -s^2\big(c+1\big)^{-1}\in \mathcal J^2\subset \mathcal J.
  \end{equation*}
  Let $u_3$ be the identity on $\big(\projone{} \wedge \projtwo{}^\perp + \projone{}^\perp \wedge \projtwo{} + \projone{}_0\vee \projtwo{}_0\big)^\perp$.
  Then $u:= u_1\oplus u_2\oplus u_3$ is unitary and $u\projtwo{}u^*=\projone{}$.
  Since $u_1$ has finite rank, we conclude that $u-1\in \mathcal J$.
\end{proof}
Property (ii) in the above proposition was obtained for $\mathcal J= \K$ in \cite[Theorem 2.7]{BL-2012-CJM}.

Using \autoref{P:in J} (i) we obtain an independent proof of the first part of \autoref{T:Arveson}  which extends it to arbitrary proper operator ideals, and establishes the sufficiency of the conditions listed.

\begin{proposition}
  \label{P:corners}
  Let $\mathcal J$ be a proper ideal of $\B(\h)$ and let $\projone{},\projtwo{}$ be projections in $\B(\h)$.
  Then the following are equivalent:
  \begin{enumerate}
  \item $\projone{}-\projtwo{}\in \mathcal J$
  \item $\projtwo{}(\projone{}-\projtwo{})\projtwo{}\in \mathcal J^2$ and $\projtwo{}^\perp(\projone{}-\projtwo{})\projtwo{}^\perp\in \mathcal J^2$
  \item The projections $\projone{} \wedge \projtwo{}^\perp$ and $\projone{}^\perp \wedge \projtwo{}$ are finite and at least one the conditions $\projtwo{}(\projone{}-\projtwo{})\projtwo{}\in \mathcal J^2$ and $\projtwo{}^\perp(\projone{}-\projtwo{})\projtwo{}^\perp\in \mathcal J^2$ holds.
  \end{enumerate}
  Furthermore if $\projone{}-\projtwo{}\in \mathscr L^2$, then
  \begin{equation*}
    [\projone{}:\projtwo{}]= \tr( \projtwo{}(\projone{}-\projtwo{})\projtwo{}+\projtwo{}^\perp(\projone{}-\projtwo{})\projtwo{}^\perp).
  \end{equation*}
\end{proposition}

\begin{proof}
  From (\ref{e:q-p Perp}) we see that
  \begin{equation}
    \label{e:corners}
    \projtwo{}(\projone{}-\projtwo{})\projtwo{}=-\projone{}^\perp \wedge \projtwo{} \oplus
    \begin{pmatrix}
      -s^2 & 0 \\
      0    & 0 \\
    \end{pmatrix}
    \text{ and } \projtwo{}^\perp(\projone{}-\projtwo{})\projtwo{}^\perp=\projone{} \wedge \projtwo{}^\perp \oplus
    \begin{pmatrix}
      0&0\\
      0&s^2\\
    \end{pmatrix}.
  \end{equation}
  Thus $\projtwo{}(\projone{}-\projtwo{})\projtwo{}\in \mathcal J^2$ (resp., $\projtwo{}^\perp(\projone{}-\projtwo{})\projtwo{}^\perp\in \mathcal J^2$) if and only if $\projone{}^\perp \wedge \projtwo{}$ is finite and $s\in \mathcal J$ (resp., $\projone{} \wedge \projtwo{}^\perp$ is finite and $s\in \mathcal J$).
  By \autoref{P:in J} (i) it is now obvious that the conditions (i)-(iii) are equivalent.

  Assume now that $\projone{}-\projtwo{}\in \mathscr L^2$, then by the implication (i) $\Rightarrow $ (ii)
  we see that
  $\projtwo{}(\projone{}-\projtwo{})\projtwo{}\in \mathcal  L^1$ and $\projtwo{}^\perp(\projone{}-\projtwo{})\projtwo{}^\perp\in \mathcal L^1$.
  Then by
  (\ref{e:corners})
  \begin{equation*}
    \tr( \projtwo{}(\projone{}-\projtwo{})\projtwo{}+\projtwo{}^\perp(\projone{}-\projtwo{})\projtwo{}^\perp)= \tr (\projone{} \wedge \projtwo{}^\perp) - \tr (\projone{}^\perp \wedge \projtwo{})
  \end{equation*}
  because
  \begin{equation*}
    \tr \begin{pmatrix}s^2&0\\0&0\end{pmatrix}= \tr \begin{pmatrix}0&0\\0&s^2\end{pmatrix}< \infty
  \end{equation*}
  as $s^2\in \mathscr L^1$.
  The conclusion then follows from \autoref{P: charact Fredh pairs}.
\end{proof}

The equivalence of just (i) and (ii) follows immediately from the identity
\begin{equation*}
  (\projone{}-\projtwo{})^2= - \projtwo{}(\projone{}-\projtwo{})\projtwo{} + \projtwo{}^\perp (\projone{}-\projtwo{})\projtwo{}^\perp
\end{equation*}
whence
\begin{equation}
  \label{e:decomp}
  ((\projone{}-\projtwo{})_{+})^2=\projtwo{}^{\perp}(\projone{}-\projtwo{})\projtwo{}^{\perp}\quad\text{and}\quad  ((\projone{}-\projtwo{})_{-})^2-=-\projtwo{}(\projone{}-\projtwo{})\projtwo{}.
\end{equation}

To conclude this survey, we observe that every Fredholm operator can be associated in a {\it natural} way to a Fredholm pair of projections $(\projone{},\projtwo{})$ so that the index of the operator equals the index of the pair.
To this end, consider any Fredholm operator $x : \h \to \mathcal{K}$ and scale $x$ to have norm 1.
After choosing an arbitrary infinite, co-infinite projection $\projtwo{}$ and identifying $\mathcal{K}$ with $\projtwo{}\h$, we have the following proposition.

\begin{proposition}
  \label{prop:digested-frame}
  Suppose $x : \h \to \projtwo{}\h$ is a contraction with $\projtwo{}$ an infinite, co-infinite projection.
  Then $x$ can be completed to an isometry $w: \h \to \h$ (i.e., $x = \projtwo{}w$), and for any such completion, if $x$ is Fredholm then $(\projone{} := ww^{*}, \projtwo{})$ is a Fredholm pair with $[\projone{}:\projtwo{}] = \idx x$.
\end{proposition}

\begin{proof}
  Consider the defect operator $(1-x^{*}x)^{\nicefrac{1}{2}}$ and a partial isometry $v$ taking $R_{1-x^{*}x}$ to $vv^{*} \le \projtwo{}^{\perp}$.
  Then define $w = x + v(1-x^{*}x)^{\nicefrac{1}{2}}$.
  Since $v^{*}x = 0 = x^{*}v$, a simple computation shows $w^{*}w = 1$, which establishes that $x$ can be completed to an isometry.

  Now suppose $w$ is any such completion, and hence $x = \projtwo{}w$.
  Define $\projone{} := ww^{*}$ and notice $x = \projtwo{}\projone{}w = (\projtwo{}\mid_{\projone{}\h})w$ when the operators are viewed on the appropriate spaces.
  Then from (\ref{e: Fredh diff spaces}) and (\ref{e: Fredh as compression}), if $x$ is Fredholm, so is $\projtwo{}\mid_{\projone{}\h}$ and
  \begin{equation*}
    \idx x = \idx(\projtwo{}\mid_{\projone{}\h})= [\projone{}:\projtwo{}]. \qedhere
  \end{equation*}
\end{proof}

\begin{remark}
  We note that any completion of a contraction $x : \h \to \projtwo{}\h$ to an isometry arises in the manner above.
  Indeed, suppose $w'$ is such a completion.
  Set $y := w'-x$ and note that $\projtwo{}y = 0$ since $x = \projtwo{}w'$.
  Thus
  \begin{align*}
    y^{*}y & = (w')^{*}w' - (w')^{*}x - x^{*}w' + x^{*}x \\
           & = 1 - 2(w')^{*}\projtwo{}w' + x^{*}x = 1- x^{*}x.
  \end{align*}
  In particular, $y = v'(1-x^{*}x)^{\nicefrac{1}{2}}$ for some partial isometry $v'$ with $(v')^{*}v' = R_{1-x^{*}x}$ and $v'(v')^{*} = R_y \le \projtwo{}^{\perp}$.
  Moreover, $u = \projtwo{} + v(v')^{*}$ is a partial isometry for which $w = uw'$.

  Another perspective of \autoref{prop:digested-frame} is that $\projone{}$ is a dilation of $xx^{*}$ to $\h$ for which $\idx x = [\projone{}:\projtwo{}]$.
  Indeed, $\projone{} := ww^{*}$ is a dilation of $xx^{*}$ because if $y = w-x$, then with respect to the decomposition $\projtwo{}+\projtwo{}^{\perp}=1$
  \begin{equation*}
    ww^{*} =
    \begin{pmatrix}
      xx^{*} & xy^{*} \\
      yx^{*} & yy^{*} \\
    \end{pmatrix}.
  \end{equation*}
\end{remark}

\subsection{Breuer Fredholm}

As mentioned in the introduction, the essential codimension/relative index of projections has found its main application in the study of spectral flows in Fredholm modules.
However, in many cases of interest the Fredholm modules are with respect to a semifinite von Neumann algebra (see \cite{CPS-2003-AM}, \cite{CP-2004-KT}, \cite{BCP+-2006-Agatoeo},\cite{CPRS-2006-AM}).
The following short summary may be of interest to the reader.

Let $M$ be a semifinite von Neumann algebra with separable predual (but not necessarily a factor), $\tau$ a faithful semifinite normal trace, and let $\mathcal{J}_\tau(M)$ the ideal of $\tau$-compact operators
\begin{equation*}
  \mathcal{J}_\tau(M):=\overline{ \spans \{x\in \M_+\mid \tau(x)< \infty\}} \quad \text {(norm closure).}
\end{equation*}
Let $\pi: M\to M/ \mathcal{J}_\tau(M)$ be the canonical quotient map and let $\|x\|_{ess}:=\|\pi(x)\|$ be the essential norm.
Then and element $x\in M$ is called $\tau$-Breuer Fredholm (also called just $\tau$-Fredholm) if $\pi(x)$ is invertible.
A necessary and sufficient condition is that
$\tau(N_x)< \infty$ (where $N_x$ is the projection on the kernel of $x$) and that there exists a projection $e\in M$ with $\tau(e)< \infty$ such that $(1-e)\h \subset x\h$.
Then the index is defined as
\begin{equation*}
\idx(x)= \tau(N_x)-\tau(N_{x^*})\in \mathbb R
\end{equation*}
and satisfies the expected properties of an index, but of course it is no longer integer valued.

The original definition by Breuer was given in terms of the ideal $\mathcal{J} (M)$ of compact operators on $M$ which received considerable attention over the years,
\begin{equation*}
  \mathcal{J} (M):=\overline{ \spans \{x\in \M_+\mid R_x \text { is finite}\}} \quad \text {(norm closure).}
\end{equation*}
When $M$ is a factor and hence has a unique trace (up to normalization), the notions of $\tau$-Breuer--Fredholm and Breuer--Fredholm coincides, but for global algebras they do not and so their theory had to be partially re-derived in \cite{BCP+-2006-Agatoeo}.

With these definitions almost all of the results listed here for $\B(\h)$ hold with the same statements and mostly with the same proofs.
So we will briefly list here only the properties that fail or that require a different proof.

\autoref {P: charact Fredh pairs} holds with the same statements and a natural modification of the proof of \cite[Proposition 3.1]{ASS-1994-JFA} in the case that $M$ is a factor, but required more work for the general case \cite[Lemma~4.1]{BCP+-2006-Agatoeo} less the trace condition which is only relevant when $M$ is a factor.

It is still true that if $\projone{},\projtwo{}\in M$ are in generic position and form a Fredholm pair then $[\projone{}:\projtwo{}]=0$, but contrary to \autoref{L:invert}, we can have $\|\projone{}-\projtwo{}\|=1$ and $g$ and $c$ are only invertible modulo $\mathcal{J}_\tau(M)$ as the following example shows.

\begin{example}
  \label{E:cpt}
  Let $M$ be a type II$_\infty$ factor.
  Let $\projtwo{}\sim \projtwo{}^\perp\sim1$ be a projection in $M$ which we decompose into a sum $\sum _1^\infty \projtwo{}_n=\projtwo{}$ of mutually orthogonal finite projections such that $\tau\big(\sum _1^\infty \projtwo{}_{2n}\big)<\infty$.
  Let
  \begin{align*}
    c     & := \sum_{1}^\infty (\sqrt{\frac{1}{2n}}\projtwo{}_{2n}+  \sum_{0}^\infty \sqrt{1-\frac{1}{2n+1}}\projtwo{}_{2n+1} \\
    s     & := \sum_{1}^\infty \sqrt{1-\frac{1}{2n}}\projtwo{}_{2n}+  \sum_{0}^\infty \sqrt{\frac{1}{2n+1}}\projtwo{}_{2n+1}  \\
    \projtwo{}     & =
    \begin{pmatrix}
      1 & 0 \\
      0 & 0 \\
    \end{pmatrix}
    \quad\text{and}\quad
    \projone{} =
    \begin{pmatrix}
      c^2 & cs  \\
      cs  & s^2 \\
    \end{pmatrix}.
  \end{align*}
  Then $\projtwo{}$ and $\projone{}$ are in generic position and $\|\projone{}-\projtwo{}\|= \|s\|=1$ while $s\in \mathcal{J}(M)$, hence $\projone{}-\projtwo{}\in \mathcal{J}(M)$, and $(\projone{},\projtwo{})$ is a Fredholm pair with respect to $M$.
\end{example}

\autoref{P:in J} (ii) holds without any changes if $M$ is a factor, but does not hold for global algebras.
Consider for instance $M:=\B(\h_1)\oplus \B(\h_2)$ and $\tau:=\tr\oplus \tr$.
Let $\projtwo{}, \projone{}$ be rank one projections in $ \B(\h_1)$ and $ \B(\h_2)$ respectively.
Then $\projtwo{}$ and $\projone{}$ are not equivalent (with respect to $M$) and hence a fortiori they are not unitarily equivalent.
On the other hand $\projone{}-\projtwo{}=\in \mathcal{J}(M)=\mathcal K(\h_1)\oplus \mathcal K(\h_2)$, hence $\projone{},\projtwo{}$ is a Fredholm pair and furthermore $[\projone{}:\projtwo{}] = \tr(\projone{})-\tr(\projtwo{})=0$.

\section{The Kadison theorem and some applications}
We begin by using the tools developed in Section \S2  to identify the integer $a-b$ in Kadison's theorem, that is, to prove \autoref{T:ess codim}.
\begin{proof}
  By (\ref{e:a-b})
  \begin{equation*}
    a-b = \tr\big( \projtwo{}\projone{}\projtwo{}- \projtwo{} + \projtwo{}^\perp \projone{} \projtwo{}^\perp \big) = \tr\big(\projtwo{}(\projone{}-\projtwo{})\projtwo{}+ \projtwo{}^\perp (\projone{}-\projtwo{})\projtwo{}^\perp\big)
  \end{equation*}
  and by (\ref{e:ab}), $\projtwo{}(\projone{}-\projtwo{})\projtwo{}  \in \mathcal L^1$ and $\projtwo{}^\perp (\projone{}-\projtwo{}) \projtwo{}^\perp \in \mathcal L^1$.
  Thus by \autoref{P:corners}, $\projone{}-\projtwo{}\in \mathcal L^2$ and $[\projone{}:\projtwo{}]= \tr( \projtwo{}(\projone{}-\projtwo{})\projtwo{}+\projtwo{}^\perp(\projone{}-\projtwo{})\projtwo{}^\perp).$
\end{proof}

As a first consequence of Kadison's theorem and of the work in Section \S 2, we observe that if the diagonal of a projection $\projone{}$ clusters sufficiently fast around $0$ and $1$ (that is, if $a+b< \infty$, or, equivalently, if $\projone{}-\projtwo{} \in \mathcal L^2$), then one can ``read'' from the diagonal the essential codimension $[\projone{}:\projtwo{}]$.
But what if $a+b= \infty$?

If $a=\infty$ and $b<\infty$, from $\projtwo{}^\perp(\projone{}-\projtwo{})\projtwo{}^\perp \in \mathcal L^1$ we can deduce that $\projone{} \wedge \projtwo{}^\perp$ is finite and $s\in \mathcal L^2$, and hence from $\projtwo{}(\projone{}-\projtwo{})\projtwo{}\not \in \mathcal L^1$ it follows that $\projone{}^\perp \wedge \projtwo{}$ is infinite.
Similarly, if $a<\infty$ and $b=\infty$ then $\projone{} \wedge \projtwo{}^\perp$ is infinite.
In either case $(\projone{},\projtwo{})$ is not a Fredholm pair and in particular, $\projone{}-\projtwo{}\not \in \K$.

Less trivial is the case when $a=b=\infty$ and $\projone{}-\projtwo{}\in \K \setminus \mathcal L^2$, as we see from the following proposition.
We first need to introduce two notations.
Given an orthonormal basis $\{e_n\}$, denote by $E$ the conditional expectation on the algebra of diagonal operators, namely $E(x)$ is the diagonal of the operator $x\in\B(\h)$.
Next, given two sequences $\xi$ and $\eta$ of non-negative numbers converging to $0$, with $\xi^*$ and $\eta^*$ their monotone non-increasing rearrangements, we say that $\xi$ is majorized by $\eta$ ($\xi\prec \eta$) if $\sum_{j=1}^n\xi^*_j\le \sum_{j=1}^n\eta^*_j$ for all $n$.

\begin{proposition}
  \label{prop:q-p not hs}
  Suppose $\projone{},\projtwo{}$ are projections with $\projone{}-\projtwo{} \in \K \setminus \mathcal{L}^2$.
  Then there exists a projection $\projone{}'$ such that $\projone{}'-\projtwo{} \in \K$, $[\projone{}':\projtwo{}] \not= [\projone{}:\projtwo{}]$ and there is an orthonormal basis $\{e_n\}$ that diagonalizes $\projtwo{}$ such that $E(\projone{})=E(\projone{}')$.
\end{proposition}

\begin{proof}
  By \autoref{P:in J} (i), $\projone{}-\projtwo{} \in \K$ implies that $\projone{} \wedge \projtwo{}^\perp$ and $\projone{}^\perp \wedge \projtwo{}$ are both finite.
  Thus $\projone{}_0-\projtwo{}_0 = (\projone{}-\projtwo{}) + (\projone{} \wedge \projtwo{}^\perp - \projone{}^\perp \wedge \projtwo{}) \in \K \setminus \mathcal{L}^2$.
  It suffices to prove the proposition for projections in generic position because then we simply set $\projone{}' := \projone{} \wedge \projtwo{} + \projone{} \wedge \projtwo{}^\perp + \projone{}_0'$ for the general case.
  So to simplify notation, assume henceforth that $\projone{},\projtwo{}$ are in generic position and have the form as in (\ref{e:p,q}).
  In particular, $\projtwo{}\sim \projtwo{}^\perp\sim 1$ and by \autoref {L:invert}, $[\projone{}:\projtwo{}]=0$.

  Next, choose a rank one projection $r'\le \projtwo{}^{\perp}$ and let $r:=\projtwo{}^{\perp} - r'$.
  After identifying $B(\h) \simeq B(\projtwo{}\h \oplus r\h \oplus \mathbb{C})$ with $M_2(\projtwo{}B(\h)\projtwo{}) \oplus \mathbb{C}$ via a partial isometry taking $\projtwo{}\h \to r\h$, consider the projection
  \begin{equation*}
    \tilde{\projone{}} :=
    \begin{pmatrix}
      c^2 & cs  & 0 \\
      cs  & s^2 & 0 \\
      0   & 0   & 1 \\
    \end{pmatrix}.
  \end{equation*}
  Note that $\tilde{\projone{}} - (0 \oplus 0 \oplus 1)$ and $\projtwo{}$ are in generic position relative to their join, and hence their essential codimension is zero.
  This implies that $[\tilde{\projone{}},\projtwo{}] = 1 \not= 0 = [\projone{},\projtwo{}]$.

  Now, choose an orthonormal basis $\{e_n\}$ that diagonalizes $\projtwo{}$, $E$ its corresponding conditional expectation, and let $\xi$ be the diagonal sequence of $s^2$.
  Then under natural notations we have
  \begin{equation*}
    E(\projone{})= E_\projtwo{}(\projone{})\oplus E_{\projtwo{}^\perp}(\projone{})
  \end{equation*}
  and furthermore, $E_\projtwo{}(\projone{})=E_\projtwo{}(\tilde{\projone{}})= E_\projtwo{}(c^2)$ and $E_{\projtwo{}^\perp}(\projone{})=E_{\projtwo{}}(s^2)= \diag \xi$.

  Since $\projone{}-\projtwo{} \in \K \setminus \mathcal{L}^2$, by \autoref{P:in J} we have that $s^2 \in \K \setminus \mathcal{L}^1$, that is $\xi\to 0$ but $\xi\not \in \ell^1$.
  By the Schur--Horn theorem for compact operators \cite[Proposition 6.4]{KW-2010-JFA} $\xi$ is majorized by the eigenvalues sequence $\lambda(s^2)$ of the operator $s^2$ and hence,
  \begin{equation*}
    \xi\prec \lambda(s^2)\prec \lambda
    \begin{pmatrix}
      s^2 & 0 \\
      0   & 1 \\
    \end{pmatrix}.
  \end{equation*}
  Now
  \(
  \begin{pmatrix}
    s^2 & 0 \\
    0   & 1 \\
  \end{pmatrix}
  \)
  is a positive compact operator with zero kernel belonging to $\B(\projtwo{}^\perp \h)$.
  Hence by \cite[Proposition 6.6]{KW-2010-JFA} there is a unitary $ u\in \B(\projtwo{}^\perp \h)$ such that
  \begin{equation*}
    \diag \xi= E_{\projtwo{}^\perp\h}\big (u
    \begin{pmatrix}
      s^2 & 0 \\
      0   & 1 \\
    \end{pmatrix}
    u^*\big).
  \end{equation*}
  Let $u':= 1\mid _{\projtwo{}\h}\oplus u$ and $\projone{}':= u'(\tilde \projone{})u'^*$.
  Then
  \begin{equation*}
    E(\projone{}')= E_\projtwo{}(c^2)\oplus \diag \xi = E(\projone{}).
  \end{equation*}
  Since $u'\projtwo{}u'^{*} = \projtwo{}$ we have
  \begin{equation*}
    [\projone{}':\projtwo{}] = [u\tilde{\projone{}}u^{*}:u\projtwo{}u^{*}] = [\tilde{\projone{}}:\projtwo{}] \not= [\projone{},\projtwo{}]. \qedhere
  \end{equation*}
\end{proof}

As a second application of \autoref{T:ess codim} and of the techniques used to prove it, we will consider a recent work by Bownik and Jasper \cite{BJ-2013-TAMS}.
Based on Kadison's characterization of diagonals of projections, Bownik and Jasper characterized the diagonals of selfadjoint operators with finite spectrum and in a key part of their analysis they too encountered an index obstruction similar to the one in \autoref{T:Kadison15}~(ii).
Following their notations, if $z\in \B(\h)$ is a selfadjoint operator with finite spectrum we let $\sigma(z)= \{a_j\}_{j=-m}^{n+r}$ and $\projone{}_j= \chi_{\{a_j\}}(z)$ be the spectral projection corresponding to the eigenvalue $a_j$, so that
\begin{equation*}
  z= \sum _{j=-m}^{n+r}a_j \projone{}_j.
\end{equation*}
For ease of notations perform if necessary a transformation so to have
\begin{equation*}
  \tr(\projone{}_j)< \infty \text{ for }j< 0 \text { and  } j>n+1, \quad a_0=0, \text{  and }  a_{n+1}=1.
\end{equation*}
Let $\{e_n\}$ be an orthonormal basis, $\{d_n\}$ be the diagonal of $z$ with respect to that basis and let as in
\autoref{T:Kadison15},
\begin{equation*}
  a = \sum_{d_n \le \nicefrac{1}{2}} d_n
  \qquad\text{and}\qquad
  b = \sum_{d_n > \nicefrac{1}{2}} (1-d_n),
\end{equation*}

Then their Theorem~4.1, which is a key component of the necessity part of their characterization, states that
\begin{theorem}[\protect{\cite[Theorem~4.1]{BJ-2013-TAMS}}]
  \label{T:BJ}
  If $a+b< \infty$ then
  \begin{enumerate}
  \item $\tr(\projone{}_j) < \infty$ for $0<j< n+1$;
  \item $a-b - \sum_{j\ne n+1} a_j\tr(\projone{}_j) \in \mathbb Z.$
  \end{enumerate}
\end{theorem}
Here of course we use the convention that $0\cdot \infty=0$ and hence $a_0\tr(\projone{}_0)=0$ whether $\tr(\projone{}_0)$ is finite or not.

We will present an independent proof of this result and at the same time identify the integer in (ii)  proving that if we set $\projtwo{}$ as in \autoref{T:Kadison15} to be the projection on $\overline{\spans}\{e_j\mid~ d_j> \nicefrac{1}{2}\}$, then
\begin{equation}
  \label{e:BJint}
  a-b - \sum_{j\ne n+1} a_j\tr(\projone{}_j)= [\projone{}_{n+1}:\projtwo{}].
\end{equation}

First we need an extension to positive elements of the equivalence of (i) and (ii) in \autoref{P:corners}.

\begin{lemma}
  \label{L: from q to x}
  Let $\mathcal J$ be a proper ideal, $x\in \B(\h)_+$ a positive contraction, and $\projtwo{}\in \B(\h)$ a projection.
  \begin{enumerate}
  \item If $\projtwo{}-\projtwo{}x\projtwo{}\in \mathcal J$ and $\projtwo{}^\perp x \projtwo{}^\perp\in \mathcal J$, then $x-\projtwo{}\in \mathcal J^{\nicefrac{1}{2}}$ and $ x\chi_{[0, \eps]}(x)\in \mathcal J$ for every $0< \eps<1$.
  \item Assume that $x$ is a projection or that $\mathcal J$ is idempotent (i.e., $\mathcal J= \mathcal J^2$).
    If $x-\projtwo{}\in \mathcal J^{\nicefrac{1}{2}}$ and $ x\chi_{[0, \eps]}(x)\in \mathcal J$ for some $0< \eps<1$, then $\projtwo{}-\projtwo{}x\projtwo{}\in \mathcal J$ and $\projtwo{}^\perp x \projtwo{}^\perp\in \mathcal J$.
  \end{enumerate}
\end{lemma}

\begin{proof}
(i).
Since $\projtwo{}^\perp x \projtwo{} x \projtwo{}^\perp\le \projtwo{}^\perp x^2  \projtwo{}^\perp\le \projtwo{}^\perp x \projtwo{}^\perp\in \mathcal J$, it follows that $\projtwo{} x \projtwo{}^\perp$ and $\projtwo{}^\perp x \projtwo{}$ belong to $\mathcal J^{\nicefrac{1}{2}}$.
But then
\begin{equation*}
  x-\projtwo{} =  (\projtwo{}x\projtwo{}- \projtwo{}) +  \projtwo{}^\perp x \projtwo{}^\perp + \projtwo{} x \projtwo{}^\perp+ \projtwo{}^\perp x \projtwo{}\in \mathcal J^{\nicefrac{1}{2}}.
\end{equation*}
Let $\eps>0$ and let $x_\eps:= x\chi_{[0, \eps]}(x)$.
Then
$0\le \projtwo{}^\perp x_\eps \projtwo{}^\perp\le \projtwo{}^\perp x \projtwo{}^\perp\in \mathcal J$, whence $\projtwo{}^\perp x_\eps \projtwo{}^\perp\in \mathcal J$.
Furthermore,
\begin{equation*}
  1-x\ge (1-x) \chi_{[0, \eps]}(x)\ge (1-\eps)  \chi_{[0, \eps]}(x)\ge (1-\eps)x_\eps
\end{equation*}
and hence
$\projtwo{}-\projtwo{}x\projtwo{}\ge (1-\eps)\projtwo{}x_\eps \projtwo{}.$
Thus $\projtwo{}x_\eps \projtwo{}\in \mathcal J$ and since
\begin{equation*}
  0\le x_\eps \le 2\big(\projtwo{}x_\eps \projtwo{} + \projtwo{}^\perp x_\eps \projtwo{}^\perp)\in \mathcal J
\end{equation*}
it follows that $x_\eps\in \mathcal J$.

(ii).
The case when $x$ is a projection is given by (\ref{e:decomp}).
Assume then that $\mathcal J$ is idempotent.
Then $x-\projtwo{}\in \mathcal J^{\nicefrac{1}{2}}= \mathcal J$ implies that $\projtwo{}-\projtwo{}x\projtwo{}=-\projtwo{}(x-\projtwo{})\projtwo{}\in \mathcal J$.
Furthermore,
$ (x-\projtwo{})^2= x^2-x\projtwo{}-\projtwo{}x+\projtwo{}$ hence
$\projtwo{}^\perp x^2 \projtwo{}^\perp= \projtwo{}^\perp (x-\projtwo{})^2 \projtwo{}^\perp\in \mathcal J$.
Then
\begin{equation*}
  \projtwo{}^\perp (x-x_\eps) \projtwo{}^\perp\le \frac{1}{\eps}\projtwo{}^\perp (x-x_\eps)^2 \projtwo{}^\perp\le \frac{1}{\eps}\projtwo{}^\perp x^2 \projtwo{}^\perp \in \mathcal J
\end{equation*}
and hence
\begin{equation*}
  \projtwo{}^\perp x \projtwo{}^\perp = \projtwo{}^\perp (x-x_\eps) \projtwo{}^\perp + \projtwo{}^\perp x_\eps \projtwo{}^\perp \in \mathcal J.
\end{equation*}
\end{proof}

Notice that if $\mathcal J$ is not idempotent and $k\in \mathcal J^{\nicefrac{1}{2}}_+\setminus \mathcal J$ is a positive contraction, then $x:=1-k$ and $\projtwo{}:=1$ satisfy both hypotheses of \autoref{L: from q to x} (ii)  but $k=\projtwo{}-\projtwo{}x\projtwo{}\not \in \mathcal J$.

Now we can proceed with the proof of \autoref{T:BJ}  and (\ref{e:BJint}).
\begin{proof}
  Set $x = \sum_{j=1}^{n+1} a_j \projone{}_j$.
  Then $0 \le x \le 1$ and
  \begin{equation*}
    z-x = \sum_{j=-m}^{-1}a_j \projone{}_j+\sum_{j=n+2}^{n+r}a_j \projone{}_j \quad\text{has finite rank.}
  \end{equation*}

  As in (\ref{e:a-b}) we have that
  \begin{equation*}
    a= \tr(\projtwo{}^{\perp} x \projtwo{}^{\perp}) + \tr(\projtwo{}^\perp (z-x) \projtwo{}^\perp) = \tr(\projtwo{}^\perp z \projtwo{}^\perp)
  \end{equation*}
  and
  \begin{equation*}
     b= \tr(\projtwo{}-\projtwo{}x\projtwo{}) - \tr(\projtwo{}(z-x)\projtwo{}) = \tr( \projtwo{}- \projtwo{}z\projtwo{}),
  \end{equation*}
  hence $\projtwo{} (z-\projtwo{}) \projtwo{} \in \mathcal L^1$ , $ \projtwo{}^\perp(z-\projtwo{})\projtwo{}^\perp \in \mathcal L^1$, and
  \begin{equation}
    \label{e:a-bBJ}
    a-b =  \tr\big( \projtwo{}(z-\projtwo{})\projtwo{}+ \projtwo{}^\perp (z-\projtwo{})\projtwo{}^\perp \big).
  \end{equation}

  We also have $\projtwo{}(x-\projtwo{})\projtwo{}\in \mathcal L^1$ and $\projtwo{}^\perp (x-\projtwo{}) \projtwo{}^\perp \in \mathcal L^1$, hence by \autoref{L: from q to x}, it follows that $x-\projtwo{}\in \mathcal L^2$ and
  \begin{equation*}
    \sum_{j=1}^n a_j \projone{}_j= x\chi_{[0, a_n]}(x) \in \mathcal L^1.
  \end{equation*}
  But then $x-\projone{}_{n+1}= \sum_{j=1}^n a_j \projone{}_j$ has finite rank and in particular, $\tr(\projone{}_j)< \infty$ for $0< j < n+1$, thus proving (i).
  As a consequence,
  \begin{equation*}
    \projone{}_{n+1}-\projtwo{}= \projone{}_{n+1}-x+ x-\projtwo{}\in \mathcal L^2
  \end{equation*}
  and hence by \autoref{P:corners},
  \begin{equation}
    \label{e:10}
    [\projone{}_{n+1}:\projtwo{}]= \tr\big(\projtwo{}(\projone{}_{n+1}-\projtwo{})\projtwo{}+ \projtwo{}^\perp(\projone{}_{n+1}-\projtwo{})\projtwo{}^\perp\big).
  \end{equation}
  Furthermore, $y:= z-\projone{}_{n+1}= \sum_{j\ne n+1} a_j \projone{}_j $ has finite rank and in particular is in $\mathcal L^1$, so that
  \begin{equation}
    \label{e:11}
    \sum_{j\ne n+1} a_j\tr(\projone{}_j)= \tr(y)= \tr(\projtwo{}y\projtwo{}+\projtwo{}^\perp y \projtwo{}^\perp).
  \end{equation}
  Finally from (\ref{e:a-bBJ}) and (\ref{e:10}),
  \begin{align*}
    a-b & =  \tr\big( \projtwo{}(\projone{}_{n+1}-\projtwo{})\projtwo{}+ \projtwo{}^\perp(\projone{}_{n+1}-\projtwo{})\projtwo{}^\perp +\projtwo{}y\projtwo{}+ \projtwo{}^\perp y \projtwo{}^\perp \big) \\
        & =[\projone{}_{n+1}:\projtwo{}]+ \tr(y).
  \end{align*}
  Thus by (\ref{e:11}), $a-b-  \sum_{j\ne n+1} a_j\tr(\projone{}_j)= [\projone{}_{n+1}:\projtwo{}]\in \mathbb Z$.
\end{proof}

\bibliographystyle{amsalpha}
\bibliography{references.bib}

\end{document}